\documentclass[11pt,a4paper,final]{article}
\usepackage{a4}
\usepackage{amsmath}%
\usepackage{amstext}%
\usepackage{amssymb}%
\usepackage{showkeys}%
\usepackage{cite}
\usepackage[latin1]{inputenc}
\usepackage{textcomp}
\usepackage{hyperref}
\usepackage{amsthm}
\usepackage{amsfonts}
\usepackage{graphicx}
\usepackage{xcolor}
\usepackage{tabularx}
\usepackage{fancyhdr}
\usepackage{multirow}

\usepackage{placeins}
\usepackage{pgfplots}
\usepackage{tikz}
\usepackage{epstopdf}
\usepackage{listings}
\usepackage[toc,page]{appendix}
\usepackage[refpage,intoc]{nomencl}
\usepackage{ifthen}
\usepackage{lipsum}
\usepackage{tabu}

\DeclareMathOperator{\eco}{e-conv}					

\DeclareMathOperator{\conv}{conv}

\DeclareMathOperator{\cl}{cl}

\DeclareMathOperator{\epco}{\rm{e}^{\prime} -conv} 	
\DeclareMathOperator{\ep}{\text{e}^{\prime}} 	
\DeclareMathOperator{\dom}{dom}
\DeclareMathOperator{\epi}{epi}
\DeclareMathOperator{\grh}{grh}

\DeclareMathOperator{\rec}{rec}

\newtheorem{definition}{Definition}[section]
\newtheorem{theorem}{Theorem}[section]
\newtheorem{remark}{Remark}[section]
\newtheorem{proposition}{Proposition}[section]
\newtheorem{corollary}{Corollary}[section]
\newtheorem{lemma}{Lemma}[section]
\newtheorem{example}{Example}[section]

\def\supp{\mathop{\rm sup}}	   	
\def\inff{\mathop{\rm inf}}

\newcommand{\en}{\rightarrow}											
\newcommand{\R}{\mathbb{R}}												
\newcommand{\N}{\mathbb{N}}												
\newcommand{\Ramp}{\overline{\R}}					

\newcommand{\ci}{\left\langle}										
\newcommand{\cd}{\right\rangle}										

\DeclareMathOperator*\czp{\cap}%

\begin{document}

\title{New Duality Results for Evenly Convex Optimization Problems}

\date{}%

\author{Maria Dolores Fajardo\thanks{University of Alicante, Alicante, Spain, md.fajardo@ua.es} \and
Sorin-Mihai Grad\thanks{Corresponding author, Faculty of Mathematics, University of Vienna, Oskar-Morgenstern-Platz 1, 1090 Vienna, Austria, sorin-mihai.grad@univie.ac.at} \and
Jos\'e Vidal\thanks{Faculty of Mathematics, Chemnitz University of Technology, 09107 Chemnitz, Germany, jose.vidal-nunez@mathematik.tu-chemnitz.de}}
\maketitle
\textbf{Abstract.} We present new results on optimization problems where the involved functions are evenly convex.
By means of a generalized conjugation scheme and the perturbation theory introduced by Rockafellar, we propose an alternative dual problem for a general optimization one defined on a separated locally convex topological space. Sufficient conditions for converse and total duality involving the even convexity of the perturbation function and $c$-subdifferentials are given. Formulae for the $c$-subdifferential and biconjugate of the objective function of a general optimization problem are provided, too.
We also characterize the total duality also by means of the saddle-point theory for a notion of Lagrangian adapted to the considered framework.\\

\maketitle

\section{Introduction}
An important part of mathematical programming from both theoretical and computational points of view is the duality theory. Rockafellar (cf. \cite{R1970}) developed the well-known perturbational approach (see also  \cite{ET1976}), consisting in the use of a perturbation function as the keystone to obtain a dual problem for a general primal one by means of the Fenchel conjugation. Both problems always verify \emph{weak duality} (the optimal value of the dual problem is less or equal to the optimal value of the primal one), as a direct consequence of the Fenchel-Young inequality, whereas conditions ensuring \emph{strong duality}(no duality gap and dual problem solvable) can be found in many references in the literature. Another related interesting problem in conjugate duality theory is the notion of \emph{converse duality}.
It corresponds to the situation where there is no duality gap and the primal problem is solvable. This issue was investigated in the convex setting in \cite{BW2006} for Fenchel duality and later extended in \cite{LFLL2009}.

In this paper we propose a new dual problem to a general primal one defined  on locally convex spaces by means of a generalized conjugation scheme and study converse duality for this primal-dual pair.
This pattern, which is inspired by a survey done by Mart\'inez-Legaz where generalized convex duality theory is applied to quasiconvex programming, is called \emph{$c$-conjugation scheme} and it was developed in \cite{MLVP2011}. In the same way like in the classical context, convexity and lower semicontinuity of the perturbation function are required in most of the regularity conditions -- see for instance \cite{B2010}, the use of the $c$-conjugation scheme is associated with the \emph{even convexity} of such a function. Evenly convex sets (functions) are a generalization of closed and convex sets (functions), see, for instance \cite{RVP2011}.
Fenchel called in \cite{F1952} a set \emph{evenly convex}, e-convex in brief, if it is the intersection of an arbitrary family, possibly empty, of open halfspaces. Such sets have been employed to study the solvability of semi-infinite linear systems containing infinitely many strict inequalities in \cite{GR2006}, whereas some important properties in terms of their sections and projections are given in \cite{KMZ2007}. Due to \cite{RVP2011}, an extended real valued function defined on a locally convex space is said to be \textit{e-convex} if its epigraph is e-convex. According to \cite{MLVP2011}, the $c$-conjugation scheme is suitable for this class of functions in the sense that the double conjugate function equals the original one if it is proper and e-convex. This result can be seen as the e-convex counterpart of the celebrated Fenchel-Moreau theorem.

The theory developed in \cite{MLVP2011} motivated in \cite{FVR2012} the generalization of some well-known properties of the sum of the epigraphs of two Fenchel conjugate functions and the infimal convolution, and, as an application, conditions for strong Fenchel duality were derived.
Later, in \cite{F2015}, the perturbation approach was used to build a dual problem by means of the $c$-conjugate duality theory, and the counterparts of some regularity conditions, i.e. conditions ensuring strong duality, from the Fenchel conjugate setting were obtained.
Moreover, in \cite{FVR2015} regularity conditions for strong duality between an e-convex problem and its Lagrange dual were established.
In \cite{FV2016SSD} we analyzed the problem of stable strong duality, and deduced Fenchel and Lagrange type duality statements for unconstrained and constrained optimization problems, respectively.
In \cite{FajardoVidal2017} the Fenchel-Lagrange dual problem of a (primal) minimization problem, whose involved functions do not need to be e-convex a priori, was derived.
Furthermore, some relations between the optimal values of the Fenchel, Lagrange and Fenchel-Lagrange dual problems were presented.
Finally, in \cite{FajardoVidal2018} we used the formulation of the Fenchel-Lagrange dual problem from \cite{FajardoVidal2017} to derive a characterization of strong Fenchel-Lagrange duality.

The purpose of the paper is twofold.
First, we analyze the fulfillment of converse duality for a primal-dual pair expressed via the perturbation function, where its even convexity will play a fundamental role.
In order to avoid repetitive arguments, we have decided to derive the aforementioned converse duality in the general form, leaving as an application the natural particularizations into the previously mentioned special cases.
As a second target, we address the problem of \emph{total duality} (no duality gap and both problems solvable) from this general perspective.
Moreover and extending in this way results from the convex setting \cite{BGB,BGS,GWB}, we provide some formulae for the $c$-subdifferential and biconjugate of the objective function of a given general optimization problem. Motivated by \cite{ET1976}, we highlight the analysis of saddle-point theory and its relation with total duality for e-convex problems through the study and application of Lagrangian functions.

The layout of this work is as follows. Section \ref{se2} contains preliminary results on e-convex sets and functions to make the paper self-contained.
Section \ref{se3} is dedicated to sufficient conditions for strong converse duality and biconjugate formulae.
Section \ref{se4} is devoted to new results on $c$-subdifferentials which allows to characterize total duality for a general primal-dual pair. Moreover, the $\varepsilon$-$c$-subdifferential of the objective function of the considered problem is expressed via the $\varepsilon$-$c$-subdifferential of the considered perturbation function. Last but not least in Section \ref{se5} we extend the saddle-point theory from the classical framework to e-convex problems, showing its close relation with total duality and we close the paper with a short section dedicated to some final remarks, conclusions and ideas for future work.

\section{Preliminaries}\label{se2}

Let $X$ be a separated locally convex space, lcs in brief, equipped with the $\sigma(X,X^*)$ topology induced by $X^*$, its continuous dual space endowed with the $\sigma(X^*,X)$ topology.
 The notation $\ci x,x^*\cd$ stands for the value at $x\in X$ of the continuous linear functional $x^*\in X^*$. Let $Y$ be another lcs. By $\R_{++}$ we denote the set of positive real numbers. For a set $D\subseteq X$ we denote its convex hull and its closure by $\conv D$ and $\cl D$, respectively.\\
According to \cite{DML2002}, a set $C\subseteq X$ is \emph{evenly convex}, e-convex in short, if for every point $x_0\notin C$, there exists $x^*\in X^*$ such that $\ci x-x_0,x^*\cd<0$, for all $x\in C$. Furthermore, for a set $C\subseteq X$, the \emph{e-convex hull} of $C$, $\eco C$, is the smallest e-convex set in $X$ containing $C$. For a convex subset  $C\subseteq X$, it always holds $C\subseteq \eco C \subseteq \cl C$.
This operator is well defined because the class of e-convex sets is closed under arbitrary intersections.
Since $X$ is a separated lcs, $X^*\neq\emptyset$. As a consequence of Hahn-Banach theorem, it also holds that $X$ is e-convex and every closed or open convex set is e-convex as well.

For a function $f:X\en \Ramp$, we denote by
	$\dom f=\{ x\in X$ : $ f(x)<+\infty \}$
the \emph{effective domain} of $f$ and by $\epi f = \left\{ \left( x,r\right) \in X\times \R \, : \, f(x)\leq r\right\}$ and $\grh f  = \left\{ \left( x,r\right) \in X\times \R \, : \, f(x)= r\right\}$ its \emph{epigraph} and its \emph{graph}, respectively. We say that $f$ is \emph{proper} if $\epi f$ does not contain vertical lines, i.e. $f(x)>-\infty$ for all $x\in X$, and  $\dom f\neq \emptyset$.
By $\cl f$ we denote the \emph{lower semicontinuous hull} of $f$, which is the function whose epigraph equals $\cl(\epi f)$.
A function $f$ is \emph{lower semicontinuous}, lsc in brief,  if for all $x\in X$, $f(x)=\cl f(x)$, and \textit{e-convex} if $\epi f$ is e-convex in $X\times \R$.
Clearly, any lsc convex function is e-convex, but the converse does not hold in general as one can see in \cite[Ex.~2.1]{FajardoVidal2017}.

The \emph{e-convex hull} of a function $f:X\rightarrow \overline{\R}$, $\eco f$, is defined as the largest e-convex minorant of $f$.
Based on the generalized convex conjugation theory introduced by Moreau \cite{Mor1970}, a suitable conjugation scheme for e-convex functions is provided in \cite{MLVP2011}.
Let us consider the space $W:=X^{\ast }\times X^{\ast }\times \R$ with \emph{the coupling functions }$c:X\times W\rightarrow \overline{\R}$ and $c^{\prime }:W\times X\rightarrow \overline{\R}$ given by
\begin{equation} \label{equation: def c and c prime}
c(x,(x^{\ast },y^{\ast },\alpha ))=c^{\prime }\left( (x^{\ast },u^{\ast},\alpha ),x\right) :=\left\{
\begin{array}{ll}
\left\langle x,x^{\ast }\right\rangle & \text{if }\left\langle x,u^{\ast}\right\rangle <\alpha, \\
+\infty & \text{otherwise.}
\end{array}
\right.
\end{equation}
Given two functions $f:X\rightarrow \overline{\R}$ and $g:W\rightarrow \overline{\R}$, the \emph{$c$-conjugate} of $f$, $f^{c}:W\rightarrow \overline{\R}$, and the \emph{c$^{\prime }$-conjugate} of $g$, $g^{c^{\prime}}:X\rightarrow \overline{\R}$, are defined
\begin{eqnarray*}
f^{c}(x^{\ast },u^{\ast },\alpha)& :=& \supp_{x\in X}\left\{ c(x,(x^{\ast},u^{\ast },\alpha ))-f(x)\right\},\\
 g^{c^{\prime }}(x) & :=& \supp_{(x^{\ast },u^{\ast },\alpha )\in W}\left\{c^{\prime }\left( (x^{\ast },u^{\ast },\alpha ),x\right) -g(x^{\ast},u^{\ast },\alpha )\right\},
\end{eqnarray*}
 with the conventions $\left( +\infty \right) +\left( -\infty \right)=\left( -\infty \right) +\left( +\infty \right) =\left( +\infty \right)-\left( +\infty \right)$ $=\left( -\infty \right) -\left( -\infty \right)=-\infty.$

Functions of the form $x \in X \en c(x,(x^{\ast},u^{\ast },\alpha))-\beta \in \Ramp$, with $(x^{\ast},u^{\ast },\alpha)\in W$ and $\beta \in \R$ are called \emph{$c$-elementary}, and, in a similar way, functions of the form $(x^{\ast},u^{\ast },\alpha)\in W \en  c(x,(x^{\ast},u^{\ast },\alpha))-\beta \in \Ramp$ with  $x \in X $ and $\beta \in \R$ are called \emph{$c^{\prime}$-elementary}.
In \cite{MLVP2011} it is shown that the familiy of proper e-convex functions from $X$ to $\R$ along with the function identically equal to $+\infty$ is actually the family of pointwise suprema of sets of $c$-elementary functions.
Similarly, a function $g:W \en \Ramp$ is \emph{e$^{\prime}$-convex} if it is the pointwise supremum of sets of $c^{\prime}$-elementary functions, and the \emph{e$^{\prime}$-convex hull} of an extended real function $g$ defined on $W$, denoted by ${\epco}g$,  is the largest e$^{\prime}$-convex minorant of it. Moreover, a set $D \subset W \times \R$ is e$^{\prime}$-convex if there exists an e$^{\prime}$-convex function $g$ such that $\epi g=D$. The e$^{\prime}$-convex hull of a set $D \subset W \times \R$, denoted by $\epco D$, is the smallest e$^{\prime}$-convex set containing $D$.

The following counterpart of the Fenchel-Moreau theorem for e-convex and e$^{\prime}$-convex functions was shown in \cite[Prop.~6.1,~Prop.~6.2,~Cor.~6.1]{ML2005}.
\begin{theorem}\label{th:Theorem2.2}
Let $f:X\en \R\cup \left\{+\infty\right\}$ and $g:W \en \Ramp$. Then
\begin{enumerate}
	\item[i)] $f^{c}$ is e$^{\prime}$-convex; $g^{c^{\prime }}$ is e-convex.
	\item[ii)] if $f$ has a proper e-convex minorant ${\eco}f=f^{c{c^{\prime}}}$; ${\epco}g=g^{{c^{\prime}}c}$.
	\item[iii)] $f$ is e-convex if and only if $f^{c c^{\prime}}=f$; $g$ is e$^{\prime}$-convex if and only if  $g^{ c^{\prime} c}=g$.
	\item[iv)] $f^{cc^{\prime}} \leq f$; $g^{c^{\prime} c} \leq g$.
\end{enumerate}
\end{theorem}
The following lemma and proposition were shown in the finitely dimensional case in \cite{GJR2003} and  \cite{RVP2011}, respectively, and can be generalized for infinitely dimensional spaces easily. Recall that the \emph{recession cone} of a nonempty subset $A \subseteq X$ is defined by $\rec A= \left \{ u \in X\, : \, a+u \in A, \text{ for all }  a \in A \right \}$.
\begin{lemma} \label{lemma:lemma 2.3}
Let $C \subseteq X$ be a nonempty e-convex set and $y \in X$ such that there exists $x_{0} \in X$ verifying $x_{0} +\lambda y \in C$, for all $\lambda \geq 0$. Then $y \in \rec C$.
\end{lemma}
\begin{proposition}\label{Proposition2.4}
Let $C \subseteq X \times \R$ be a nonempty e-convex set such  that $(0,1) \in \rec C$. Then  $h(x)= \inf \left \lbrace a \in \R \,:\,(x,a) \in C \right \rbrace$  is an e-convex function and $\epi h=C \cup \grh h$.
\end{proposition}
For convenience in Section 3, we need to ensure the fact that certain e-convex sets in  $X \times \R$ with $(0,1) \in \rec C$, must be epigraphs of e-convex functions.
This requirement will be fulfilled via the following definition.
\begin{definition}
A non-empty e-convex set $C \subseteq X \times \R$ with $(0,1) \in \rec C$ is \emph{functionally representable} if $\grh h \subseteq C$, for $h(x)= \inf  \{ a \in \R \,:$ $(x,a) \in C \}$.
\end{definition}
\begin{remark}\label{rem:Remark2.6}
A non-empty e-convex set $C \subseteq X \times \R$ with $(0,1) \in \rec C$ is functionally representable if $\epi h = C$, where  $h(x)= \inf \left \lbrace a \in \R \,:\,(x,a) \in C \right \rbrace$.
\end{remark}

\section{Converse Duality and Biconjugation}
\label{se3}

Let $\Phi:X\times Y \rightarrow \overline{\R}$ be a perturbation function for $(GP)$ and $0\in\Pr_Y (\dom \Phi)$.
Its $c$-conjugate $\Phi ^{c}:(X^* \times Y^*)\times (X^* \times Y^*)  \times \R \rightarrow \overline{\R}$, is defined by
\begin{equation*}
\Phi ^{c}(( x^*,y^*),( u^*,v^*),\alpha) =\sup_{( x,y) \in X \times Y}\left\{ \bar{c}%
\left( \left( x,y\right),\left( \left( x^{\ast },y^{\ast }\right),\left(
u^{\ast },v^{\ast }\right),\alpha \right) \right) -\Phi \left( x,y\right)
\right\},
\end{equation*}%
where $\bar{c}:(X \times Y)\times (X^* \times Y^*)\times (X^* \times Y^*)\times \R\rightarrow
\overline{\R}$ is%
\begin{equation} \label{equation:def c bar}
\overline{c}\left( \left( x,y\right) ,\left( \left( x^{\ast },y^{\ast
}\right) ,\left( u^{\ast },v^{\ast }\right) ,\alpha \right) \right) =\left\{
\begin{array}{ll}
\left\langle x,x^{\ast }\right\rangle +\left\langle y,y^{\ast }\right\rangle
& \text{if }\left\langle x,u^{\ast }\right\rangle +\left\langle y,v^{\ast
}\right\rangle <\alpha , \\
+\infty & \text{otherwise.}%
\end{array}%
\right.
\end{equation}%
Also in this point we recall the associated coupling function to $\bar c$, which will be used in Section 3.2, $\bar{c}^\prime: (X^* \times Y^*)\times (X^* \times Y^*)\times \R \times(X \times Y)\rightarrow \overline{\R}$ ,
\begin{equation} \label{equation:def c bar prime}
\bar{c}^\prime (((x^*, y^*), (u^*,v^*) ,\alpha),(x,y))=\bar c ((x,y),((x^*, y^*), (u^*,v^*) ,\alpha)).
\end{equation}
Proceeding along the lines of \cite{FVR2012}, we consider the following primal-dual pair of problems, which verify weak duality,%
\begin{equation*}
	\begin{aligned}
		(GP)&\inf & \Phi(x,0) \nonumber\\
		& s.t. & x\in X,
	\end{aligned}
	\hspace{1cm}
	\begin{aligned}
		(GD_{c})   \sup \hspace{0.3cm}& \big\{-\Phi ^{c}\left( \left( 0,y^{\ast }\right), \left( 0,v^{\ast }\right) ,\alpha \right)\big\}& \\
		 s.t.\hspace{0.3cm} & y^*,v^*\in Y ^*, \alpha >0.&
	\end{aligned}
\end{equation*}
This dual problem can also be expressed via the \emph{infimum value function} $p:Y \rightarrow \overline{\R}$, defined by
$p\left( y\right) :=\inf_{x\in X}\Phi \left( x,y\right)$, in the following way
\begin{eqnarray*}
(GD_{c}) & \sup & \{ -p^{c}\left( y^*,v^*,\alpha \right)\}\\
& s.t. & y^{\ast },v^{\ast }\in Y ^{\ast }, \alpha >0.
\end{eqnarray*}%

From $(GD_{c})$ one can derive the minimization (primal) problem
\begin{eqnarray*}
(\overline{GP}_c) &  \inf & \Phi ^{c}\left( \left( 0,y^{\ast }\right)
,\left( 0,v^{\ast }\right) ,\alpha \right) \\
& s.t. & y^{\ast },v^{\ast }\in Y ^{\ast }, \alpha >0.
\end{eqnarray*}
and under appropiate regularity conditions, the optimal value of $(\overline{GP}_c)$ is equal to the optimal value of its dual, with the last problem being solvable as well. Let us calculate this dual problem, which we denote by $(\overline{GD})$.
If we consider the biconjugate function $\Phi^{c c^{\prime}}:X\times Y \rightarrow \overline{\R}$, then, for all  $(x,y)\in X\times Y,~((x^*,y^*),(u^*,v^*))\in X^*\times Y^* $ and  $\alpha \in \R$
\begin{eqnarray*}
	\Phi^{c c^{\prime}}(x,y)& + & \Phi^c((x^*,y^*),(u^*,v^*),\alpha)\geq  c ((x,y),((x^*,y^*),(u^*,v^*), \alpha)).
\end{eqnarray*}
Taking $y=0, ~x^*=u^*=0$ and  $\alpha>0 $, $\Phi^{c c^{\prime}}(x,0)) + \Phi^c((0,y^*),0,v^*),\alpha)\geq  0$,
\begin{eqnarray}\label{E1}
 \inff_{\substack{y^*,v^*\in Y^*,\\\alpha>0}}  \Phi^c((0,y^*),(0,v^*),\alpha)\geq \supp_{\substack{x\in X}} \{-\Phi^{c c^{\prime}}(x,0)\}\geq \supp_{\substack{x\in X}} \{-\Phi(x,0)\}.
\end{eqnarray}

We consider the following dual problem for  $(\overline{GP}_c)$
\begin{eqnarray*}
(\overline{GD}) &\sup & \{-\Phi(x,0)\} \nonumber\\
& s.t. & x\in X.%
\end{eqnarray*}%

It follows that $v(\overline{GD})\leq v(\overline{GP}_c)$ and, moreover, strong duality for $(\overline{GP}_c)-(\overline{GD})$ is equivalent to converse duality for $(GP)-(GD{_c}).$
In  \cite{F2015}, via the e-convexity of the perturbation function, some regularity conditions for a general primal problem $(GP)$ and its general dual $(GD_{c})$ were obtained.
In particular, assuming the properness and e-convexity of $\Phi$, the closedness-type one,
\begin{equation*}\label{eq:C5}\tag{C5}
	{\Pr}_{W\times\R}\left(\epi\Phi ^{c}\right) \text{ is e}^{\prime}\text{-convex},
\end{equation*}
was introduced to guarantee strong duality for the primal-dual pair $(GP)-(GD_c)$. Due to \cite[Lem.~5.3]{F2015} this condition can be expressed (under the mentioned hypotheses) as
${\Pr}_{W\mathbb{\times R}}\left(\epi\Phi ^{c}\right)=\epi\Phi(\cdot,0)^{c}$. 

In order to be able to provide a counterpart of \eqref{eq:C5} for the pair $(\overline{GP}_c)-(\overline{GD})$, we need first a suitable perturbation function for the new primal problem that plays the role of $\Phi$. Let us consider $\Psi:=\Phi^{c}$ and the perturbation variable space to be $X^*\times X^*$.
$\Psi$ is always e$^{\prime }$-convex, an important feature which represents a difference from the standard context of strong duality on e-convex problems -- see \cite{F2015,FV2016SSD,FajardoVidal2018}.
Recall that there the perturbation function is assumed to be e-convex in order to use the mentioned regularity condition.
Naming $G:=\Phi^{c}((0,\cdot),(0,\cdot),\cdot)$ the new objective function, and defining the space $W_Y:=Y^*\times Y^*\times \R$, $G:W{_Y}\en \Ramp$, one rewrites $(\overline{GP}_c)$ as
\begin{eqnarray*}
(\overline{GP}_c) &  \inf & G(y^*, v^*, \alpha) \\
& s.t. & y^{\ast },v^{\ast }\in Y ^{\ast }, \alpha >0.
\end{eqnarray*}
Observe that if $\Psi((0,y^*),(0,v^*),\alpha) := G(y^*, v^*, \alpha)$, for all $(y^*, v^*, \alpha)\in W{_Y}$, this means that the function $\Psi$ can be considered as a perturbation function for $(\overline{GP}_c)$  and, consequently, a dual problem for it would be
\begin{eqnarray*}
(\overline{GD_c}) &\sup & \{-\Psi^{c^{\prime}}(x,0)\} \nonumber\\
& s.t. & x\in X.
\end{eqnarray*}
In the case $\Phi$ is e-convex, $\Psi^{c^{\prime}}=\Phi$, and we obtain
\begin{eqnarray*}
(\overline{GD_c}) &\sup & \{-\Phi(x,0)\} \nonumber\\
& s.t. & x\in X.
\end{eqnarray*}

Since the biconjugate function is important to develop converse duality, let us analyze further properties of the function $\Phi^{cc^\prime}$.

\begin{theorem}\label{th51}
It always holds $(\Phi (\cdot, 0))^{cc^{\prime}}\geq \Phi^{cc^{\prime}}(\cdot, 0)$.
If $\Phi$ is proper and e-convex and \eqref{eq:C5} is satisfied, one obtains $(\Phi (\cdot, 0))^{cc^{\prime}}= \Phi^{cc^{\prime}}(\cdot, 0)$.
\end{theorem}

\begin{proof}
Let us take $(x^*,u^*, \alpha)\in W$. Then,
$$(\Phi (\cdot, 0))^c (x^*, u^*, \alpha) = \sup\limits_{x\in X} \{ c(x, (x^*, u^*, \alpha)) - \Phi (x, 0)\}$$
$$= \sup\limits_{x\in X} \{ \bar c((x,0), ((x^*, 0), (u^*,0), \alpha)) - \Phi (x, 0)\}$$ $$ \leq \sup\limits_{x\in X, y\in Y} \{ \bar c((x, y), ((x^*,0),( u^*,0), \alpha)) - \Phi (x, y)\}$$
\begin{equation}\label{E2}
= \Phi^c ((x^*,0),( u^*,0), \alpha) \leq  \inf\limits_{y^*, v^*\in Y^*} \Phi^c ((x^*, y^*), (u^*, v^*), \alpha).
\end{equation}

Taking now the $c^{\prime}$-conjugate of the first and last function in \eqref{E2}, we obtain,
$$(\Phi (\cdot, 0))^{cc^{\prime}} (\cdot)  \geq  \left(\inf\limits_{y^*, v^*\in Y^*} \Phi^c ((\cdot, y^*), (\cdot, v^*), \alpha)\right)^{c^{\prime}} (\cdot)$$
	$$ = \sup_{\substack{x^*, u^* \in X^*\\ y^*, v^* \in Y^*, \alpha \in \R}} \{c(\cdot, (x^*, u^*, \alpha)) - \Phi^c ((x^*, y^*), (u^*, v^*), \alpha))\}$$
	$$ = \sup\limits_{\substack{x^*, u^*\in X^*,\\ y^*, v^*\in Y^*,\ \alpha \in \R}} \{\bar c^{\prime}(((x^*, y^*), (u^*, v^*), \alpha), (\cdot, 0)) - \Phi^c ((x^*, y^*), (u^*, v^*), \alpha))\}$$
\begin{equation}\label{E3} = \Phi^{cc^{\prime}}(\cdot, 0).
\end{equation}

When $\Phi$ is proper and e-convex and \eqref{eq:C5} holds, one has by \cite[Prop.~5.4]{F2015} strong duality for the primal-dual pair $(GP)-(GD_c)$, i.e., for all $ (x^*, u^*, \alpha)\in W$, $(\Phi (\cdot, 0))^c (x^*, u^*, \alpha) = \min_{y^*, v^*\in Y^*} \Phi^c ((x^*, y^*), (u^*, v^*), \alpha)$. The chain in \eqref{E3} is then a chain of equalities, and one obtains $(\Phi (\cdot, 0))^{cc^{\prime}}= \Phi^{cc^{\prime}}(\cdot, 0)$.
\end{proof}

Next we present the counterparts of Lemma 5.1, 5.2 and 5.3 from \cite{F2015} that are necessary for proving the converse duality statement.


\begin{lemma} \label{lemma biconj1}
It holds
\begin{equation*}
\Psi((0,\cdot),(0,\cdot), \cdot)^{c{^{\prime}}}\leq \inf_{x \in X} \Psi ^{c ^{\prime}} (x, \cdot).
\end{equation*}
\begin{proof}
Fixing $y \in Y$, for all $x \in X, ((x^*,u^*),(y^*,v^*)) \in X^* \times Y^*$ and $\alpha \in \R$, we have
$$c{^{\prime}}(((x^*,u^*),(y^*,v^*),\alpha), (x,y))-\Psi ((x^*,u^*),(y^*,v^*),\alpha) \leq \Psi ^{c^{\prime}} (x,y).$$

Let $ x ^* =y^*=0$. Then, for all $x\in X, (u^*,v^*,\alpha) \in Y^* \times Y^* \times \R$,
$c{^{\prime}}((u^*,v^*,\alpha), y)-\Psi ((0,u^*),(0,v^*),\alpha) \leq \Psi ^{c^{\prime}} (x,y)$.
Then, $\Psi((0,\cdot),(0,\cdot),$ $\cdot)^{c^{\prime}}(y)\leq \inf_{x \in X} \Psi ^{c ^{\prime}} (x, y)$.
\end{proof}
\end{lemma}

We continue with a Moreau-Rockafellar-type result involving the $c^{\prime}$-conjugate of the function $\Psi((0,\cdot)$, $(0,\cdot),\cdot)$ that is also of interest per se.

\begin{lemma}\label{lemma 3.2}
$\Psi((0,\cdot),(0,\cdot),\cdot)^{c^{\prime}}$ is the e-convex hull of $\inf_{x\in X} \Psi^{c^{\prime}} (x, \cdot)$.
\end{lemma}
\begin{proof}
We have, for all $(y^*,v^*,\alpha)\in Y^*\times Y^*\times \R$,
$$
(\inf_{x\in X} \Psi^{c^{\prime}}(x,\cdot))^{c}(y^*,v^*,\alpha)  = \sup_{y\in Y}\left\{ c(y,(y^*,v^*,\alpha))-\inf_{x\in X}\Psi^{c^{\prime}}(x,y)\right\}$$
$$ = \sup_{(x,y)\in X \times Y}\left\{ c((x,y),((0,y^*),(0,v^*),\alpha))-\Psi^{c^{\prime}}(x,y)\right\}$$
$$ = \Psi^{c^{\prime} c}((0,y^*),(0,v^*),\alpha)= \Psi((0,y^*),(0,v^*),\alpha),$$
due to the $\ep$-convexity of $\Psi$ and Theorem \ref{th:Theorem2.2}.
Taking into account that, applying again Theorem \ref{th:Theorem2.2} it yields $\eco(\inf_{x\in X} \Psi^{c^{\prime}}(x,\cdot))=(\inf_{x\in X} \Psi^{c^{\prime}}(x,\cdot))^{c{c^{\prime}}}$, we conclude that $\eco(\inf_{x\in X} \Psi^{c^{\prime}}(x,\cdot))=\Psi((0,\cdot),(0,\cdot),\cdot)^{c^{\prime}}$.
\end{proof}

\begin{lemma}\label{lemma 3.3}
If the  e-convex hull of ${\Pr}_{Y\times\R}(\epi\Psi^{c^{\prime}})$ is functionally representable, then $\epi \Psi$ $((0,\cdot), (0,\cdot),\cdot)^{c^{\prime}}$ is the e-convex hull of ${\Pr}_{Y\times\R}(\epi\Psi^{c^{\prime}})$.
\end{lemma}
\begin{proof}
First, we see that  ${\Pr}_{Y\times\R}(\epi\Psi^{c^{\prime}})\subset \epi \Psi((0,\cdot),(0,\cdot),\cdot)^{c^{\prime}}$.
Take any point $(y,\beta)\in {\Pr}_{Y\times\R}(\epi\Psi^{c^{\prime}})$. It means that there exists $x\in X$ such that $\Psi^{c^{\prime}}(x,y)\leq \beta$, and, for all $(x^*, y^*),(u^*,v^*)\in X^* \times Y^*$ and $\alpha \in \R$, we have
 \begin{equation*}
c^{\prime}(((x^*, y^*),(u^*,v^*),\alpha), (x,y))-\Psi((x^*, y^*),(u^*,v^*),\alpha) \leq \beta.
\end{equation*}
Take $x^*=u^*=0$.
Then, for all $(y^*,v^*,\alpha) \in Y^* \times Y^* \times \R$,
 \begin{equation*}
c^{\prime}(( y^*,v^*,\alpha),y)-\Psi((0,y^*), (0,v^*),\alpha) \leq \beta.
\end{equation*}
We have $\Psi ((0, \cdot),(0, \cdot), \cdot)^{c^{\prime}} (y) \leq \beta$ and $(y,\beta)\in\epi \Psi((0,\cdot),(0,\cdot),\cdot)^{c^{\prime}}$.
Hence,
\begin{equation}\label{ecuacion 3.1}
 \eco{\Pr}_{Y\times\R}(\epi\Psi^{c^{\prime}})\subset \epi \Psi((0,\cdot),(0,\cdot),\cdot)^{c^{\prime}}.
\end{equation}
On the other hand, let us observe that $(y, \beta)\in \epi \Psi((0,\cdot),(0,\cdot),\cdot)^{c^{\prime}}$  if and only if
$$ \sup_{(y^*,v^*,\alpha)\in  Y^*\times Y^*\times \R} \left\{ c^{\prime} ((y^*,v^*,\alpha), y)-\Psi ((0,y^*),(0,v^*),\alpha)\right\} \leq \beta,$$
or, recalling the definition of the infimum value function $p$,
\begin{equation*}
\sup_{(y^*,v^*,\alpha)\in  Y^*\times Y^*\times \R} \left\{ c^{\prime} ((y^*,v^*,\alpha), y)- p^{c}(y^*,v^*,\alpha) \right\} \leq \beta,
\end{equation*}
which means that $p^{cc^{\prime}} (y) \leq \beta$.
Thus, by Theorem \ref{th:Theorem2.2} it follows
 \begin{equation}\label{ecuacion 3.2}
 \epi \Psi((0,\cdot),(0,\cdot),\cdot)^{c^{\prime}}=\epi (\eco p).
\end{equation}
Now, let $C=\eco ( {\Pr}_{Y\times\R}(\epi\Psi^{c^{\prime}}))$.
Taking any point $(x,y, \Psi^{c^{\prime}}(x,y)) \in \epi  \Psi^{c^{\prime}}$, we have $(y, \Psi^{c^{\prime}}(x,y)) \in C$ and, for all $\lambda \geq 0$,
$(y, \Psi^{c^{\prime}}(x,y))+\lambda(0,1)=(y, \Psi^{c^{\prime}}(x,y)+\lambda) \in C$. According to Lemma \ref{lemma:lemma 2.3}, $(0,1)\in \rec C$, and since $C$ is functionally representable by assumption, due to Remark \ref{rem:Remark2.6} we obtain $C=\epi h$ with
$h(x)= \inf \left \lbrace a \in \R \,:\,(x,a) \in C \right \rbrace$ .
If we show that $h=\eco p$, we have, using also \eqref{ecuacion 3.2}, the chain
\begin{equation*}
 \epi \Psi((0,\cdot),(0,\cdot),\cdot)^{c^{\prime}}=\epi (\eco  p)=\epi h = C = \eco ( {\Pr}_{Y\times\R}(\epi\Psi^{c^{\prime}})).
\end{equation*}
Take some $y \in \dom h$ such that $h (y) > - \infty$. Combining \eqref{ecuacion 3.1} with \eqref{ecuacion 3.2}, we deduce that
$ C\subseteq \epi (\eco p).$
Moreover, for all $a \in \R$ such that  $(y,a) \in C$, we have $\eco p(y) \leq a$, so $  \eco p(y) \leq  h(y)$.

In the case $h (y) = - \infty$, there exists a sequence $ \left \lbrace a_{r} \right \rbrace \subset \R, a_{r} < -{1 /r}$, for all $r \in \N$, verifying $(y, a_{r}) \in C$, for all $r \in \R$, hence
 $(y, a_{r}) \in \epi (\eco p)$ and  $\eco p(y)=-\infty$.

Finally, we show that for all $y \in \dom p$, $h(y) \leq p(y)$.
Notice that from the e-convexity of $h$ and the definition of the e-convex hull of a function, we deduce $h(y) \leq \eco p(y)$.

Let $y \in \dom p$, with $p (y) > - \infty$, denote $a=p(y)$. Then $\inf_{x \in X} \Phi (x,y)= a$, and there exists a sequence $ \left \lbrace {x_{r}} \right \rbrace \subset X$, such that $\lim_{r \rightarrow +\infty} \Phi (x_{r},y)=a$, and
\begin{equation*}
\left ( y, \Phi (x_{r},y) \right ) \in  {\Pr}_{Y\times\R}(\epi\Phi) \subseteq {\Pr}_{Y\times\R}(\epi\Psi^{c^{\prime}}) \subseteq C,
\end{equation*}
for all $r \in \N$, then $h(y) \leq \Phi (x_{r},y)$, for all $r \in \N$, and $h(y) \leq a$.\\
In the case $p (y) = - \infty$, there exists a sequence $ \left \lbrace x_{r} \right \rbrace \subset X$ such that $\Phi (x_{r},y) < -{1 /r}$, for all $r \in \N$, and, since  $\left ( y, \Phi (x_{r},y) \right ) \in C$, for all $r \in \N$, we have $h(y) \leq \Phi (x_{r},y)$, for all $r \in \N$, concluding  $h (y) = - \infty$.
\end{proof}

We give the following sufficient condition to converse duality, assuming that $\Phi$ is a proper and e-convex function,
\begin{equation}\label{eq:C5_barra}
	\tag{$\overline{{\rm C}5}$}
	{\Pr}_{Y\times\mathbb{ R}}\left(\epi\Psi^{c^{\prime}}\right) \text{ is e-convex and functionally representable}.
\end{equation}

\begin{remark}
Although under e-convexity of the perturbation function $\Phi$, it holds $\Psi^{c^{\prime}}=\Phi$, we kept in the formulation of \eqref{eq:C5_barra} the function  $\Psi^{c^{\prime}}$, in order to have a similar formulation to \eqref{eq:C5}.
\end{remark}

\begin{proposition}\label{conv-d}
If $\Phi$ is proper and e-convex, the regularity condition \eqref{eq:C5_barra} guarantees the converse duality between $(GP)$ and $(GD)$.
\end{proposition}
\begin{proof}
According to Lemma \ref{lemma 3.3}, under \eqref{eq:C5_barra} one has
\begin{equation}
{\Pr}_{Y\times\R}(\epi\Psi^{c^{\prime}})=\epi \Psi((0,\cdot),(0,\cdot),\cdot)^{c^{\prime}}.\label{ecuacion_8}
\end{equation}
On the other hand, it is clear that
\begin{equation}
{\Pr}_{Y\times\R}(\epi\Psi^{c^{\prime}})\subset \epi\inf_{x\in X} \Psi^{c^{\prime}}(x,\cdot), \label{ecuacion_10}
\end{equation}
since, taking  $(\bar y,\bar \beta)\in {\Pr}_{Y\times\R}(\epi\Psi^{c^{\prime}})$,  there exists $\overline x\in X$ such that $\Psi^{c^{\prime}}(\bar x,\bar y)\leq \bar \beta$, then $\inf_{x\in X}\Psi^{c^{\prime}} (x,\bar y)\leq \bar \beta$.
Now, we show that
\begin{equation}
\inf_{x\in X} \Psi^{c^{\prime}}(x,\cdot)=\Psi((0,\cdot),(0,\cdot),\cdot)^{c^{\prime}}.\label{ecuacion_11}
\end{equation}
Due to Lemma \ref{lemma biconj1}, it is enough to show that, for all $y \in Y$, one has
\begin{equation}
\inf_{x\in X} \Psi^{c^{\prime}}(x,y)\leq \Psi((0,\cdot),(0,\cdot),\cdot)^{c^{\prime}}(y).\label{ecuacion_12}
\end{equation}
Take any point  $y \in Y$ and denote $a=\Psi((0,\cdot),(0,\cdot),\cdot)^{c^{\prime}}(y)\in \R$.
Then, according to \eqref{ecuacion_8} it yields that $(y,a) \in {\Pr}_{Y\times\R}(\epi\Psi^{c^{\prime}})$.
Hence, by \eqref{ecuacion_10}, $(y,a) \in \epi\inf_{x\in X} \Psi^{c^{\prime}}(x,\cdot)$ and \eqref{ecuacion_12} holds.
Now, let $\gamma=v(\overline{GP}_c)$, and let us observe that if $\alpha \leq 0$, $(y^*,v^*,\alpha)\notin \dom (\Psi(0,\cdot),(0,\cdot),\cdot)$, for any $(y^*,v^*)\in Y^*\times Y^*$.
To show the latter statement, it suffices to take any point $x \in \dom F$ and check that it holds $c((x,0),(0, y^*),(0,v^*), \alpha)-\Phi(x,0)=+\infty$. Then
\begin{eqnarray*}
-\gamma& = & \sup_{(y^*,v^*)\in Y^*\times Y^*, \alpha >0} \big\{-\Psi ((0,y^*),(0,v^*),\alpha)\big\}\\
& = & \sup_{(y^*,v^*)\in Y^*\times Y^*, \alpha >0} \left\{c(0,(y^*,v^*,\alpha))-\Psi((0,y^*),(0,v^*),\alpha) \right\}\\
&=& \Psi((0,\cdot),(0,\cdot),\cdot))^{c^{\prime}} (0) =\inf_{x\in X} \Psi^{c^{\prime}}(x,0),
\end{eqnarray*}
where the last equality comes from \eqref{ecuacion_11}. Since the perturbation function is e-convex, $\Psi^{c^{\prime}}=\Phi$, and we have
$-\gamma =\inf_{x\in X} \Phi(x,0)=-\sup_{x\in X} \{-\Phi(x,0)\}=-v(\overline{GD})=-v(\overline{GP}_c)$. Moreover and due to $(0,-\gamma)\in \epi\Psi((0,\cdot),(0,\cdot),\cdot)^{c^{\prime}}$, equation \eqref{ecuacion_8} implies that there exists $\bar x \in X$ such that $\Phi (\bar x, 0) \leq -\gamma$ and $-\Phi (\bar x, 0) \geq v(\overline{GD})$, so $(\overline{GD})$ is solvable.
\end{proof}

Next proposition states an alternative formulation for condition \eqref{eq:C5_barra} .

\begin{proposition}\label {prop biconj2}
If $\Phi$ is proper and e-convex, and $ \eco{{\Pr}_{Y\times\R}(\epi\Psi^{c^{\prime}})}$ is functionally representable, then  ${\Pr}_{Y\times\R}(\epi \Psi ^{c ^{\prime}} )$ is e-convex if and only if
\begin{equation*}
 \Psi((0,\cdot),(0,\cdot), \cdot)^{c{^{\prime}}}= \min_{x \in X} \Psi ^{c ^{\prime}} (x, \cdot).
\end{equation*}
\end{proposition}
\begin{proof}
According to Lemma \ref{lemma 3.3} and taking also in consideration Lemma \ref{lemma biconj1}, ${\Pr}_{Y\times\R}(\epi \Psi ^{c ^{\prime}} )$ is e-convex if and only if
${\Pr}_{Y\times\R}(\epi \Psi ^{c ^{\prime}} )=\epi \Psi((0,\cdot),(0,\cdot),\cdot)^{c^{\prime}}.$
In particular, this means that for all $y \in \dom \Psi ((0, \cdot), (0, \cdot), \cdot) ^{c ^{\prime}}$, there exists $\bar x \in X$ such that
$ \Psi ^{c ^{\prime}}   (\bar x, y) = \Psi((0,\cdot),(0,\cdot), \cdot)^{c{^{\prime}}}(y)$.
By Lemma \ref{lemma biconj1}, this is equivalent to
$ \Psi ^{c ^{\prime}}  (\bar x,y)= \min_{x \in X} \Psi ^{c ^{\prime}} (x, y)= \Psi((0,\cdot),(0,\cdot), \cdot)^{c^{\prime}}(y)$.
\end{proof}
The following corollary arises from Lemma \ref{lemma biconj1} and Proposition \ref{prop biconj2}.
\begin{corollary}\label{th52} It always holds
$$\left(\inf\limits_{x\in X}\Phi (x, \cdot)\right)^{cc^{\prime}}\leq \inf\limits_{x\in X} \Phi^{cc^{\prime}}(x, \cdot).$$
Moreover, if $\Phi$ is proper and e-convex, and $(\overline{C5})$ holds, then
$$\left(\inf\limits_{x\in X}\Phi (x, \cdot)\right)^{cc^{\prime}} = \min\limits_{x\in X} \Phi^{cc^{\prime}}(x, \cdot).$$
\end{corollary}
\begin{remark}
Despite the similarity between converse duality and strong duality, see for instance \cite{FV2016SSD}, the property of being functionally representable makes a difference between them. It is totally necessary since the e-convex envelope of a given epigraph is not, in general, an epigraph anymore.
\end{remark}

\section{$C$-subdifferentiability}
\label{se4}

\subsection{New results and application to total duality}

The subdifferentiability of a function at a point associated with the $c$-conjugation scheme was considered in \cite{MLVP2011} as follows.

\begin{definition}\label{def:c-subgradient}
Let $f:X\en \Ramp$ be a function. A vector $ (x^*, u^*, \alpha) \in W$ is a \textit{$c$-subgradient} of $f$ at $x_{0} \in X$ if  $f(x_{0}) \in \R, \langle x_{0}, u^* \rangle < \alpha$ and, for all $x \in X$,
\begin{equation*}
f(x)-f(x_{0}) \geq c(x,(x^*, u^*, \alpha))-c(x_{0}, (x^*, u^*, \alpha)).
\end{equation*}
The set of all the $c$-subgradients of $f$ at $x_{0}$ is denoted by $\partial_{c}f(x_{0})$ and is called the \textit{$c$-subdifferential set} of $f$ at $x_0$. In the case $f(x_{0}) \notin \R$, it is set $\partial_{c}f(x_{0})=\emptyset$.
\end{definition}

We also denote by $\partial$ the classical (convex) subdifferential. We state now the counterpart of \cite[Prop.~5.1 (Ch.~I)]{ET1976} for $c$-subdifferentials.

\begin{lemma} \label{lemma 5}
Let $f:X\en \Ramp$ be a function and $x_{0} \in \dom f$. Then  $ (x^*,u^*, \alpha) \in  \partial_{c}f(x_{0})$ if and only if $\langle x_{0}, u^* \rangle<\alpha$ and
$f(x_{0}) +f^{c} (x^*,u^*, \alpha)= c(x_{0},(x^*,u^*, \alpha))$.
\end{lemma}
\begin{proof}
Let $ (x^*, u^*, \alpha) \in  \partial_{c}f(x_{0})$.
This means $\langle x_{0},u^* \rangle<\alpha$ and
\begin{equation*}
f^{c} (x^*, u^*, \alpha) =\sup_{ x\in X} \{c(x,(x^*, u^*, \alpha))-f(x)\}\leq c(x_{0},(x^*, u^*, \alpha))- f(x_{0}).
\end{equation*}
This means that
$f(x_{0}) +f^{c} (x^*, u^*, \alpha)\leq c(x_{0},(x^*, u^*, \alpha))$,
which concludes the proof, since the opposite inequality is always true.
\end{proof}

We extend properties of subdifferentiability to $c$-subdifferentiability, introducing previously the definition of the $c^{\prime}$-subgradient of a function $g\colon W\to\Ramp$. Next definition allows us to generalize a well-known result for convex and lsc functions  -- see, for instance,  \cite[Th.~2.4.4]{Z2002} -- for e-convex functions. Note also that in \cite{VPV} the (convex) subdifferential of an e-convex function is considered.

\begin{definition}\label{def:cprime-subgradient}
Let $g:W\en \Ramp$ be a function. Then, $ x \in X$ is a \textit{$c^{\prime}$-subgradient} of $g$ at $(x^*_0,u^*_0,\alpha_0) \in W$ if  $g(x) \in \R, \langle x, u_0^* \rangle < \alpha$ and, for all $(x^*,u^*,\alpha) \in W$,
\begin{equation*}
	g(x^*,u^*,\alpha)-g(x^*_0,u^*_0,\alpha_0) \geq c^{\prime}((x^*, u^*, \alpha),x)-c^{\prime}((x^*_0, u^*_0, \alpha_0),x).
\end{equation*}
The set of all the $c^{\prime}$-subgradients of $g$ at $(x^*_0,u^*_0,\alpha_0)$ is denoted by $\partial_{c^{\prime}}g(x^*_0,u^*_0,\alpha_0)$.
 In the case $g(x^*_0,u^*_0,\alpha_0) \notin \R$, it is set $\partial_{c^{\prime}}g(x^*_0,u^*_0,\alpha_0)=\emptyset$.
\end{definition}

We present the counterpart of Lemma \ref{lemma 5} for functions $g\colon W\to \Ramp$.

\begin{lemma} \label{lemma:lemma4_5}
Let $g\colon W\to\Ramp$ be a function and $(x^*_0,u^*_0,\alpha_0)\in\dom g$.
Then $x\in\partial_{c^{\prime}}g(x^*_0,u^*_0,\alpha_0)$ if and only if $\ci x,u^*_0\cd <\alpha_0$ and
\begin{equation}\label{eq:Equality_cprime_subdifferential}
	g(x^*_0,u^*_0,\alpha_0) + g^{c^{\prime}}(x) = c^{\prime}((x^*_0,u^*_0,\alpha_0),x).
\end{equation}
\end{lemma}
\begin{proof}
Let $x\in\partial_{c^{\prime}} g(x^*_0,u^*_0,\alpha_0)$.
This amounts to saying that $\ci x,u^*_0\cd < \alpha_0$ and
\begin{equation*}
	c^{\prime} ((x^*_0,u^*_0,\alpha_0),x) - g(x^*_0,u^*_0,\alpha_0) \geq c^{\prime}((x^*,u^*,\alpha),x) - g(x^*,u^*,\alpha)
\end{equation*}
for all $(x^*,u^*,\alpha)\in W$.
Then, this is equivalent to saying that
$$
	c^{\prime} ((x^*_0,u^*_0,\alpha_0),x) - g(x^*_0,u^*_0,\alpha_0) \geq
\sup_{(x^*,u^*,\alpha)\in W} \{c^{\prime}((x^*,u^*,\alpha),x) - g(x^*,u^*,\alpha)\} = g^{c^{\prime}}(x).
$$
Taking into account that the converse inequality always holds, we get (\ref{eq:Equality_cprime_subdifferential}).
\end{proof}

Next is the counterpart to \cite[Cor.~23.5.1]{R1970} via the c-conjugation scheme.

\begin{proposition} \label{prop:Proposition4_6}
Let $f\colon X\to \Ramp$ and $x \in \dom f$. If $(x^*,u^*,\alpha) \in \partial_c f(x)$ then $ x \in \partial_{c^{\prime}} f^c(x^*,u^*,\alpha)$
and the converse statement holds if $f$ is e-convex.
\end{proposition}
\begin{proof}
Let $(x^*,u^*,\alpha)\in\partial_c f(x)$. By Lemma \ref{lemma 5}, $\ci x,u^*\cd < \alpha$ and,
\begin{equation}\label{eq:firstequality}
	f (x) + f^c(x^*,u^*,\alpha) = c(x_0,(x^*,u^*,\alpha)).
\end{equation}
As $c(x,(x^*,u^*,\alpha)) = c^{\prime}((x^*,u^*,\alpha),x)$ and $(f^c)^{c^{\prime}} \leq f$, we get
\begin{equation}\label{eq:secondequality}
	(f^c)^{c^{\prime}}(x) + f^c(x^*,u^*,\alpha) \leq c^{\prime}((x^*,u^*,\alpha),x)
\end{equation}
which means that $x \in \partial_{c^{\prime}} f^c(x^*_0,v^*_0,\alpha_0)$ according to Lemma \ref{lemma:lemma4_5}.
Now, in the case $f$ is e-convex, we have $(f^c)^{c^{\prime}}= f$, hence (\ref{eq:firstequality}) and (\ref{eq:secondequality}) are equivalent.
\end{proof}

We apply these results to \textit{total duality} for $(GP)-(GD{_c})$, that is
\begin{equation*}
\min_{x \in X} \Phi (x,0)=\max_{y^*, v^* \in X, \alpha>0} \{- \Phi ^{c}  (((0,y^*),0,v^*),\alpha)\},
\end{equation*}
i.e. the situation when both the primal and the dual have optimal solutions and their optimal values coincide. In the classical setting (see \cite {ET1976}), total duality for $(GP)-(GD)$ and finiteness of both optimal values amounts to the existence of a point $(\bar x, \bar y^*) \in X \times Y^*$ verifying
$(0, \bar y^*) \in \partial \Phi ( \bar x,0)$, or, equivalently, see \cite[Th.~23.5]{R1970},
$\Phi(\bar x,0)+\Phi^*(0, \bar y ^*)=0$,
 being, in that case, $\bar x$ an optimal solution of $(GP)$ and $\bar y^*$ an optimal solution of $(GD)$.

\begin{proposition} \label{proposition 4.3}
Let $\Phi:X\times Y\en\Ramp$, $\bar x \in X$ and $(\bar y^*, \bar v^*, \bar \alpha) \in Y^* \times Y^* \times \R$. Then $((0,\bar y^*),(0,\bar v^*),\bar \alpha) \in \partial_{c} \Phi (\bar x, 0)$ if and only if $\bar x$ is an optimal solution to $(GP), (\bar y^*, \bar v^*, \bar \alpha)$ a solution to $(GD_{c})$ and $v(GP)=v(GD_{c}) \in \R$.
\end{proposition}
\begin{proof}
If $((0,\bar u^*),(0,\bar v^*),\bar \alpha) \in \partial_{c} \Phi (\bar x, 0)$, then an application of Lemma  \ref {lemma 5} implies
$\Phi(\bar x,0) $ $+\Phi ^{c} ((0,\bar y^*),(0,\bar v^*),\bar \alpha)$ $=0$, with $\bar\alpha >0$. Further
\begin{equation}
v(GP) \leq  \Phi(\bar x,0)=-\Phi ^{c} ((0,\bar y^*),(0,\bar v^*),\bar \alpha) \leq v(GD_c) \leq v(GP), \label{ecuacion_15}
\end{equation}
in such a way that  $v(GP)=v(GD_c)\in \R$ and both problems are solvable.
Assuming \eqref{ecuacion_15} true, we get
$\Phi(\bar x,0) +\Phi ^{c} ((0,\bar y^*),(0,\bar v^*),\bar \alpha) =0$
and, as $(\bar y^*, \bar v^*, \bar \alpha)$ is an optimal solution of $(GD{_c})$, it holds $\bar\alpha >0$. Hence, by Lemma \ref{lemma 5}, $((0,\bar y^*),(0,\bar v^*),\bar \alpha) \in \partial_{c} \Phi (\bar x, 0)$.
\end{proof}
\begin{remark}
Notice that for $\Phi$ e-convex, a necessary and sufficient condition for total duality for $(GP)-(GD_c)$ is (by Propositions \ref{prop:Proposition4_6} and \ref{proposition 4.3}) the existence of  $\bar x \in X, (\bar y^*, \bar v^*, \bar \alpha) \in Y^* \times Y^* \times \R$ verifying  $(\bar x, 0) \in  \partial_{c^{\prime}} \Phi (\bar y^*, \bar v^*, \bar \alpha)$.
\end{remark}

\subsection{$\varepsilon$-c-subdifferentiability}

Next we extend the characterizations of $\varepsilon$-subdifferential formulae for convex functions from \cite{BGS} to the current e-convex setting. First recall the definition of the $\varepsilon$-$c$-subdifferential of a function $f:X\en \Ramp$ from \cite[Def.~4]{FVR2012}.
\begin{definition}
	A vector $ (x^*, u^*, \alpha) \in W$ is an \textit{$\varepsilon$-$c$-subgradient} of $f$ at $\bar x \in X$ if  $f(\bar x) \in \R, \langle \bar x, u^* \rangle < \alpha$ and, for all $x \in X$,
\begin{equation}
f(x)-f(\bar x) \geq c(x,(x^*, u^*, \alpha))-c(\bar x, (x^*, u^*, \alpha))-\varepsilon.\label{definition eps-c-subgradient}
\end{equation}
The set of all the $\varepsilon$-$c$-subgradients of $f$ at $\bar x$ is denoted by $\partial_{c,\varepsilon}f(\bar x)$ and is called the \textit{$\varepsilon$-$c$-subdifferential set} of $f$ at $\bar x$. If $f(\bar x) \notin \R$, take $\partial_{c, \varepsilon}f(\bar x)=\emptyset$.
\end{definition}
Note that \eqref{definition eps-c-subgradient} amounts to $f(x_0)+f^c(x^*,u^*,\alpha) \leq c(x_0,(x^*,u^*,\alpha))+\varepsilon$.

In \cite[Th.~3.1]{BGS} the $\varepsilon$-subdifferential of the objective function of $(GP)$ is characterized by means of the $\varepsilon$-subdifferential of the considered perturbation function via a lsc regularity condition. Its e-convex counterpart follows.

\begin{theorem}\label{th53} Let $\Phi$ be e-convex. For all $\varepsilon \geq 0$ and all $x\in X$,
$$\partial _{c,\varepsilon}\Phi (\cdot, 0)(x)=\czp_{\eta >0} {\Pr}_{X^*\times X^*\times \R} (\partial_{c,\varepsilon + \eta}\Phi (x, 0))$$
holds if and only if the function $\inf_{y^*, v^*\in Y^*}\Phi^c ((\cdot, y^*), (\cdot, v^*), \cdot)$ is e$^{\prime}$-convex.
\end{theorem}

\begin{proof}  $\Phi (\cdot, 0)^c$ is the $e^{\prime}$-convex hull of $\inf_{y^*, v^*\in Y^*}\Phi^c((\cdot, y^*), (\cdot, v^*), \cdot)$, according to \cite[Lem.~5.2]{F2015}. Consequently,
the fact that $\inf_{y^*, v^*\in Y^*}\Phi^c ((\cdot, y^*), \cdot, v^*), \cdot)$ is e$^{\prime}$-convex can be equivalently characterized by the inclusion
$\epi((\Phi(\cdot, 0))^c)\subseteq \epi \big(\inf_{y^*, v^*\in Y^*}\Phi^c  ((\cdot, y^*), (\cdot, v^*), \cdot))$, too. Note also that one has
$$\partial _{c,\varepsilon}\Phi (\cdot, 0)(x) \supseteq \czp_{\eta >0} {\Pr}_{X^*\times X^*\times \R} (\partial_{c,\varepsilon + \eta}\Phi (x, 0))$$ even without taking $\Phi$ to be e-convex.

In first place, take an arbitrary pair $((x^*, u^*, \alpha), r)\in \epi(\Phi(\cdot, 0)^c)$, i.e. $\Phi(\cdot, 0)^c (x^*, u^*, \alpha)\leq r$. Let $x\in \dom (\Phi (\cdot, 0))$ and $\varepsilon= r+ \Phi (x, 0)- c(x, (x^*, u^*, \alpha))$ $\geq 0$.
Then $(\Phi(\cdot, 0))^c(x^*, u^*, \alpha) + \Phi (x, 0)\leq c(x, (x^*, u^*, \alpha)) + \varepsilon,$
 i.e. $(x^*, u^*, \alpha)\in \partial _{c,\varepsilon}\Phi (\cdot, 0)(x)$. Using the hypothesis, whenever $\eta >0$ there exist
$y^*_{\eta}, v^*_{\eta}\in Y^*$ for which $((x^*, y^*_{\eta}), (u^*, v^*_{\eta}), \alpha)\in \partial_{c,\varepsilon + \eta}\Phi (x, 0)$. Fixing $\eta>0$, we get
\begin{equation*}
	\Phi (x, 0) + \Phi^c((x^*, y^*_{\eta}), (u^*, v^*_{\eta}), \alpha) \leq
{\bar c}((x, 0), ((x^*, y^*_{\eta}), (u^*, v^*_{\eta}), \alpha)) + \varepsilon + \eta,
\end{equation*}
followed by
$$\Phi (x, 0) + \inf_{y^*, v^*\in Y^*}\Phi^c ((x^*, u^*), (y^*, v^*), \alpha) \leq c(x, (x^*, u^*, \alpha)) + \varepsilon + \eta,\quad \forall \eta > 0.$$
Letting $\eta$ tend towards 0 and taking into consideration the value of $\varepsilon$, it follows
$$\Phi (x, 0) + \inf_{y^*, v^*\in Y^*}\Phi^c ((x^*, u^*), (y^*, v^*), \alpha) \leq r+ \Phi (x, 0),$$
i.e.$((x^*, u^*, \alpha), r)\in  \epi \big(\inf_{y^*, v^*\in Y^*}\Phi^c ((\cdot,y^* ), (\cdot, v^*), \cdot)\big)$.

To show the converse statement,  let $\varepsilon \geq 0$ and $x\in X$. If $\Phi (x, 0)=+\infty$ one has $\partial _{c,\varepsilon}\Phi (\cdot, 0)(x)=\partial_{c,\varepsilon + \eta}\Phi (x, 0)= \emptyset$ for all $\eta >0$.

Assume further that $\Phi (x, 0)\in \R$. For $((x^*, u^*, \alpha), r)\in \partial_{c,\varepsilon}\Phi (\cdot, 0)(x)$, one has
$\Phi(\cdot, 0)^c(x^*, u^*, \alpha) + \Phi (x, 0)\leq c(x, (x^*, u^*, \alpha)) + \varepsilon.$
Since the hypothesis means that $\Phi(\cdot, 0)^c =\inf_{y^*, v^*\in Y^*}\Phi^c ((\cdot, y^*), (\cdot, v^*), \cdot)$, one obtains
$$\Phi (x, 0) + \inf_{y^*, v^*\in Y^*}\Phi^c ((x^*, y^*), (u^*, v^*), \alpha)  \leq c(x, (x^*, u^*, \alpha)) + \varepsilon.$$
Moreover, fixing $\eta >0$ there exist $y^*_{\eta}, v^*_{\eta}\in Y^*$ such that
$$\Phi (x, 0) + \Phi^c ((x^*,y^*_{\eta}), ( u^*, v^*_{\eta}), \alpha) \leq {\bar c}((x, 0), ((x^*, y^*_{\eta}), (u^*, v^*_{\eta}), \alpha)) + \varepsilon + \eta,$$
i.e. $((x^*, y^*_{\eta}),(u^* , v^*_{\eta}), \alpha)\in \partial_{c,\varepsilon + \eta}\Phi (x, 0)$, which yields the conclusion.
\end{proof}

A simpler formula for the $\varepsilon$-subdifferential of the objective function of $(GP)$ by means of the $\varepsilon$-subdifferential of the considered perturbation function can be found in \cite[Th.~4.1]{BGS}, under a stronger regularity condition. Its e-convex counterpart follows, extending \cite[Th.~11]{FVR2012} for general optimization problems.

\begin{theorem}\label{th54} Let $\Phi$ be e-convex. For all $\varepsilon \geq 0$ and all $x\in X$,
$$\partial _{c,\varepsilon}\Phi (\cdot, 0)(x)={\Pr}_{W} (\partial_{c,\varepsilon}\Phi (x, 0))$$
holds if and only if \eqref{eq:C5} is fulfilled.
\end{theorem}

\begin{proof} By \cite[Lem.~5.3]{F2015}, the condition \eqref{eq:C5} is equivalent to the inclusion $\epi((\Phi(\cdot, 0))^c)\subseteq {\Pr}_{W}\epi( \Phi^c)$ (since the opposite one is always valid). Note also that $\partial _{c,\varepsilon}\Phi (\cdot, 0)(x) \supseteq {\Pr}_{W} (\partial_{c,\varepsilon}\Phi (x, 0))$
in general.
To show the direct statement, take some $((x^*, u^*, \alpha), r)\in \epi(\Phi(\cdot, 0)^c)$ and $x\in \dom (\Phi (\cdot, 0))$, and let $\varepsilon= r+ \Phi (x, 0)- c(x, (x^*, u^*, \alpha)) \geq 0$.
Analogously to the proof of Theorem \ref{th53}, one obtains
$y^*_{\varepsilon}, v^*_{\varepsilon}\in Y^*$ for which $((x^*,y^*_{\varepsilon}), ( u^*, v^*_{\varepsilon}),$ $\alpha)\in \partial_{c,\varepsilon}\Phi (x, 0)$, i.e.
$$\Phi (x, 0) + \Phi^c((x^*, y^*_{\varepsilon}), (u^*, v^*_{\varepsilon}), \alpha) \leq c((x, 0), ((x^*, y^*_{\varepsilon}), (u^*, v^*_{\varepsilon}), \alpha)) + \varepsilon.$$ Employing the value of $\varepsilon$, one gets
$(((x^*, y^*_{\varepsilon}), (u^*, v^*_{\varepsilon}), \alpha), r)\in  \epi \Phi^c$, which yields \eqref{eq:C5}.

Now, let $\varepsilon \geq 0$ and $x\in X$. If $\Phi (x, 0)=+\infty$ one has $\partial _{c,\varepsilon}\Phi (\cdot, 0)(x)=\partial_{c,\varepsilon}\Phi (x, 0)= \emptyset$.
Let further $\Phi (x, 0)\in \R$. For $(x^*, u^*, \alpha), r)\in \partial_{c,\varepsilon}\Phi (\cdot, 0)(x)$, one has
$(\Phi(\cdot, 0))^c(x^*, u^*, \alpha) + \Phi (x, 0)\leq c(x, (x^*, u^*, \alpha)) + \varepsilon.$
By \cite[Prop.~5.4]{F2015}, \eqref{eq:C5} implies
$\Phi(\cdot, 0)^c =\min\limits_{y^*, v^*\in Y^*}\Phi^c ((\cdot, y^*), (\cdot, v^*), \cdot)$.
Consequently for every $(x^*, u^*, \alpha)\in W$ there are some $y^*, v^* \in Y^*$ such that
$$\Phi^c ((x^*, y^*), (u^*, v^*), \alpha) + \Phi (x, 0)\leq {\bar c} ((x,0), (x^*, y^*), (u^*, v^*), \alpha)+ \varepsilon,$$
i.e. $((x^*, y^*), (u^*, v^*), \alpha) \in \partial_{c,\varepsilon}\Phi (x, 0)$, which yields the conclusion.
\end{proof}

\section{Saddle-point theory on e-convex problems}
\label{se5}


In the classical setting there exists a connection between saddle-point theory and total duality.
This relation comes due to the fact that saddle-points can be characterized in terms of optimal solutions for the primal and the dual problem -- see \cite[Sect.~3.3]{ET1976}.
In the following we extend the definition of Lagrangian function and saddle-point theory into the application of the $c$-conjugation scheme.
The following definitions are the counterpart of Definitions 3.1 and 3.2 in \cite{ET1976}, respectively. For more on Lagrangian functions in the classical (convex) case we refer the reader to \cite[Sect.~3.3]{BGW}.

\begin{definition}\label{def:c-Lagrangian}
The function $L:X \times (Y^* \times Y^* \times \R_{++}) \en \Ramp$ defined by
\begin{equation*}
L(x, (y^*, v^*, \alpha) )=\inf_{y \in \ Y_x} \lbrace  \Phi( x,y)-c(y,  (y^*, v^*, \alpha)) \rbrace
\end{equation*}
where $Y_x=\dom \Phi(x, \cdot)$, for each $x \in X$,  is called \emph {the $c$-Lagrangian} of the problem $(GP)$ relative to $\Phi.$
\end{definition}
In the classical setting, the Lagrangian function of  $(GP)$ relative to $\Phi$, $L:X\times Y^*  \en \Ramp$,
 \begin{equation*}
L(x, y^* )=\inf_{y \in  Y} \lbrace  \Phi( x,y)-\ci y,y^*\cd \rbrace,
\end{equation*}
verifies that $L_x :Y^* \en \Ramp$, defined for all $x\in X$ by $L_x (y^*)=L(x,y^*)$, is a concave and upper semicontinuous function.
Nevertheless, the function $L_{y^*} :X \en \Ramp$ given by $L_{y^*}(x)=L(x,y^*)$ for all $y^* \in Y^*$ is convex when $\Phi$ is convex -- see \cite[Sect.~3.3]{ET1976}.
In our context, the function $L_x :Y^* \times Y^* \times \R_{++}\en \Ramp$, which is defined for all $x\in X$ by
\begin{equation*}	
	L_x (y^*, v^*, \alpha)=L(x,(y^*, v^*, \alpha)),
\end{equation*}
verifies that $-L_x=\Phi(x,\cdot)^{c}$ and it is e$^{\prime}$-convex.
However, we can not guarantee that for all $(y^*, v^*, \alpha) \in Y^* \times Y^* \times \R_{++}$, the function $L_{(y^*, v^*, \alpha)} :X \en \Ramp$ given by $L_{(y^*, v^*, \alpha)}(x)=L(x,(y^*, v^*, \alpha)) = -\Phi(x,\cdot)^{c}(y^*, v^*, \alpha)$ is convex when $\Phi$ is so.

\begin{example}
Let us consider $\left( P\right) ~ \inf_{x\in \R}\left\{ f\left( x\right) +g\left( Ax\right)\right\} $, where $f:\R \en \R$, $g:\R \en \Ramp$ and $A:\R \en \R$ are defined as
\begin{equation*}
f(x)=x, ~ g(y)=\delta_{]-\infty, 0]} (y), ~ A(z)=z,
\end{equation*}
with the perturbation function $\Phi(x,y)= f(x)+g(Ax+y)$.
Take the point $(y^*, v^*, \alpha)=(1,1,1) \in \R ^3$ and let us calculate $L_{(1,1,1)}:\R \en \Ramp$.
\begin{equation*}
L_{(1,1,1)}(x)=\inf_{y \leq -x}\lbrace{x-c(y,(1,1,1))} \rbrace=\left\{
\begin{array}{ll}
\inf_{y \leq -x}\lbrace {x-y}\rbrace & \text{if }x>-1 , \\
-\infty & \text{otherwise.}
\end{array}
\right.
\end{equation*}
Taking into account that, if $x>-1$ and $y \leq -x$ then $x-y \geq ~ x-1 >-2$, we conclude
\begin{equation*}
L_{(1,1,1)}(x)=\left\{
\begin{array}{ll}
-2 & \text{if }x>-1 , \\
-\infty & \text{otherwise},
\end{array}
\right.
\end{equation*}
which is not convex.
\end{example}
\begin{definition}
A point $(\bar x, (\bar y^*, \bar v^*, \bar \alpha))\in X \times ( Y^* \times Y^* \times \R_{++})$ is called a \emph{saddle-point} of $L$ if
\begin{equation*}
L(\bar x,  (y^*, v^*, \alpha)) \leq L(\bar x, (\bar y^*, \bar v^*, \bar \alpha) ) \leq L( x,  (\bar y^*, \bar v^*, \bar \alpha)),
\end{equation*}
holds for all  $ x \in X$ and $ ( y^*, v^*,  \alpha)\in  Y^* \times Y^* \times \R_{++}.$
\end{definition}
The $c$-Lagrangian of the problem $(GP)$ relative to $\Phi$ is related to both optimal values in the following way.


\begin{proposition}
One always has
\begin{eqnarray}
v(GD_c) & = \sup\limits_{\substack{y^*, v^* \in Y^*\\ \alpha >0}} \inf\limits_{x \in X} L(x, (y^*, v^*, \alpha)). \label {ecuacion 16}
\end{eqnarray}
If $\Phi(x, \cdot)$ is e-convex for all $x \in X$, it also holds
\begin{eqnarray}
v(GP) & =  \inf\limits_{x \in X}\sup\limits_{\substack{y^*, v^* \in Y^*\\ \alpha >0}}  L(x, (y^*, v^*, \alpha)). \label {ecuacion 17}
\end{eqnarray}
\end{proposition}
\begin{proof}

Take any point $((x^*, y^*), (u^*, v^*), \alpha) \in (X^* \times Y^*) \times (X^* \times Y^*) \times \R_{++}$, then
\begin{equation*}
\Phi ^{c} ((x^*, y^*), (u^*, v^*), \alpha)=\sup_{(x,y) \in \dom \Phi} \lbrace c((x,y), (x^*, y^*), (u^*, v^*), \alpha) - \Phi (x,y) \rbrace \\
\end{equation*}
and letting $x^*=0, u^*=0$,
\begin{align*}
-\Phi ^{c} ((0, y^*),(0,  v^*), \alpha) &=\inf_{(x,y) \in \dom \Phi} \lbrace \Phi (x,y)-c(y,  (y^*, v^*, \alpha)) \rbrace \\
&= \inf_{x \in X}\lbrace { \inf_{y \in  Y_{x}} \lbrace { \Phi (x,y)-c(y,  (y^*, v^*, \alpha))}} \rbrace.
\end{align*}
Then
\begin{equation}
-\Phi ^{c} ((0, y^*),(0,  v^*), \alpha)=     \inf_{x \in  X} L(x, (y^*, v^*, \alpha))   \label {ecuacion_18}
\end{equation}
and \eqref{ecuacion 16} holds.
On the other hand, since  $\Phi(x,\cdot)$ is e-convex, we have
\begin{align*}
\Phi (x,0)&=\inf_{x \in X} \Phi(x,\cdot)^{cc^{\prime}}(0)\\
&= \sup_{\substack{y^*, v^* \in Y^*\\ \alpha\in \R}} \lbrace c^{\prime} ((y^*, v^*, \alpha), 0)-\Phi(x,\cdot)^c (y^*, v^*, \alpha) \rbrace \\
&=  \sup_{\substack{y^*, v^* \in Y^*\\ \alpha\in \R}} \lbrace c^{\prime} ((y^*, v^*, \alpha), 0)-\sup_{y \in Y_x}\lbrace {c(y,(y^*, v^*, \alpha)-\Phi (x,y)}\rbrace \rbrace \\
&= \sup_{\substack{y^*, v^* \in Y^*\\ \alpha>0}} \lbrace c^{\prime} ((y^*, v^*, \alpha), 0)+  L(x, (y^*, v^*, \alpha) ) \rbrace.
\end{align*}
We obtain
\begin{equation}
\Phi (x,0)= \sup_{\substack{y^*, v^* \in Y^*\\ \alpha>0}}  L(x, (y^*, v^*, \alpha) ) \label{ecuacion_19}
\end{equation}
and \eqref{ecuacion 17} also holds.
\end{proof}


\begin{proposition} \label {proposition 5.5}
Let us assume that $\Phi (x, \cdot)$ is e-convex, for all $x \in X$.
Then, $(\bar x, (\bar y^*, \bar v^*, \bar \alpha))\in X \times ( Y^* \times Y^* \times \R_{++})$ is a saddle-point of $L$ if and only if $\bar x$ is a optimal solution of  $(GP), ~(\bar y^*, \bar v^*, \bar \alpha)$ is a optimal solution of $(GD_c)$ and $v(GP)=v(GD_c) \in \R$.
\end{proposition}
\begin{proof}
Let $(\bar x, (\bar y^*, \bar v^*, \bar \alpha))\in X \times ( Y^* \times Y^* \times \R_{++})$ be a saddle-point of $L$.
Taking into account \eqref{ecuacion_18} and \eqref{ecuacion_19}, we obtain
\begin {eqnarray*}
-\Phi ^{c} ((0, \bar y^*),(0,  \bar v^*), \bar \alpha)) &=& \inf_{x \in X}  L(x, (\bar y^*, \bar v^*, \bar \alpha) ) = L(\bar x, (\bar y^*, \bar v^*, \bar \alpha) )\\
&=&\sup_{y^*, v^* \in Y^*, \alpha>0}  L(\bar x, (y^*, v^*, \alpha) )=\Phi (\bar x, 0).
\end{eqnarray*}
Then, in particular $\Phi(\bar x, 0) +\Phi ^{c} ((0,\bar y^*),(0,\bar v^*),\bar \alpha) \leq 0$ and, applying Lemma \ref{lemma 5}, we have that $((0,\bar y^*),(0,\bar v^*),\bar \alpha)$ $\in \partial_{c} \Phi (\bar x, 0)$, which finishes the proof in virtue of Proposition \ref{proposition 4.3}.

For the converse statement, since $v(GP)=\Phi (\bar x,0)$, $v(GD_c)=-\Phi ^{c} ((0,\bar y^*),$ $(0,\bar v^*),\bar \alpha)$ and $v(GP)=v(GD_c) \in \R$ we obtain, using \eqref{ecuacion_18} and \eqref{ecuacion_19} that
\begin{equation*}
 L(\bar x, (y^*, v^*, \alpha) ) \leq \Phi (\bar x,0)=-\Phi ^{c} ((0,\bar y^*),(0,\bar v^*),\bar \alpha) \leq  L(x, (\bar y^*, \bar v^*, \bar \alpha) ).
\end{equation*}
Hence, $\Phi (\bar x,0)= L(\bar x, (\bar y^*, \bar v^*, \bar \alpha) )$ and $(\bar x, (\bar y^*, \bar v^*, \bar \alpha) )$ is a saddle-point of $L$.
\end{proof}
\begin{remark}
In a more general framework, Penot and Rubinov in \cite{PR2005} related the Lagrangian and the pertubational approach (or parametrization approach, as it is named in that paper) to duality for optimization problems. In order to allow a more comprehensive comparison between their work and ours, we have adapted their notation to the one used in this work.
Given a set $Z$, a Lagrangian for $(GP)$ is a function $L:X\times Z \en \Ramp$ which must verify that $F( \cdot)=\sup_{z \in Z}L(\cdot,z)$, in which case, the optimal value of $(GP)$ satisfies
\begin{equation*}
 v(GP)=\inf_{x \in X} \sup_{z \in Z} L(x,z).
\end{equation*}
Note that no convexity or topological assumptions were imposed on the involved functions in this case. Defining a dual functional for $(GP)$, as $d_L(z)=\inf_{x \in X}L(x,z)$, for every $z \in Z$, a dual problem for $(GP)$ is
\begin{equation*}
 (GD_{L}) ~ \sup_{z \in Z} d_L(z),
\end{equation*}
and for this primal-dual pair of optimization problems there holds weak duality. According to \cite[Prop.~1 (Sect.~3.3)]{PR2005}, if $\Phi:X \times Y \en \Ramp$ is a perturbation function for $(GP)$, and, for all $x \in X$, $\Phi (x, \cdot)$ is $H_c$-convex at $0$ (see \cite[Sect.~2]{PR2005} for a definition), then
\begin{equation} \label{defLagraPenot}
L(x,z)=\inf_{y \in Y} \lbrace { \Phi(x,y)-c(y,z)} \rbrace
\end{equation}
is a Lagrangian for $(GP)$.
Here $c: Y\times Z \en \Ramp$ is any coupling function. As it can be observed, the function $L$ in Definition \ref{def:c-Lagrangian} is not the same as (\ref{defLagraPenot}): if we take the infimum on $Y$, the function in Definiton \ref{def:c-Lagrangian} is always $-\infty$, except perhaps for points  $(\bar x, (0, 0, \bar \alpha))\in X \times ( Y^* \times Y^* \times \R_{++})$, because of the special structure of the coupling functions we considered.
\end{remark}

\section{Final remarks, conclusions and future work}
\label{section:Conclusions}

In this paper we present new results regarding evenly convex (e-convex) functions, in particular converse and total duality statements for e-convex problems, that extend their counterparts from the (classical) convex case. Other results can be generalized to the current setting as well, for instance the $\varepsilon$-duality statements from \cite{GRA},
however the proofs work straightforwardly and present no difficulty so we opted not to include them here. On the other hand, some results known at the moment for proper, convex and lower semicontinous functions, such as the maximal monotonicity of their subdifferentials or the fact that their proximal point operators are single-valued, do not hold in general for e-convex functions -- check for instance the function considered in \cite[Ex.~2.1]{FajardoVidal2017}.

We extend in this article the notion of converse duality from the convex setting to e-convex optimization problems, providing a sufficient closedness-type regularity condition for it and an alternative formulation using the infimum value function. In order to prove the mentioned results we introduced the notion of functionally representable functions and we also gave a new Moreau-Rockafellar type result for e$^{\prime}$-convex functions. We introduced the concept of the $c^{\prime}$-subdifferential of a function,
providing novel characterizations of the elements of $c$-subdifferentials and $c^{\prime}$-subdifferentials, respectively, and studying how total duality is connected with them.
Formulae for the $c$-subdifferential and biconjugate of the objective function of a given general optimization problem are provided, too.
On the other hand, we extend the definition of the classical Lagrangian towards the e-convex setting by means of the $c$-conjugation scheme and relate the corresponding saddle-points to total duality. The results for general optimization problems can be specialized for constrained and unconstrained optimization problems as well.

We have restricted ourselves to closedness type regularity conditions, but taking into consideration \cite[Sect.~3]{FV2016SSD} one can alternatively provide interiority type ones to the same end as well. Using the connection between the $c$- and $c^{\prime}$-subdifferentials and the notion of $c$-conjugation, other paths can be followed along this direction. On the other hand, the investigations from \cite{VPV} on subdifferentials of e-convex functions can be continued in the vein of this paper, too. Last but not least other properties of proper, convex and lower semicontinous functions and formulae involving them could be extendable to the e-convex setting.\\


\textbf{Acknowledgements.} Research partially supported by MINECO of Spain and ERDF of EU, Grant MTM2014-59179-C2-1-P, Austrian Science Fund (FWF), Project M-2045, and German Research Foundation (DFG), Project GR3367/4-1. The work of S.-M. Grad was partially carried out during a research visit at the Erwin Schr\"odinger International Institute for Mathematics and Physics in Vienna in the framework of the Programme Modern Maximal Monotone Operator Theory\_LFB\_2019

\bibliographystyle{plain}
\bibliography{biblio}

\end{document}